\documentclass[reqno,11pt]{amsart}

\usepackage{xcolor}
\usepackage{graphicx}
\usepackage[hidelinks]{hyperref}
\usepackage{amscd}

\usepackage{amsmath}
\usepackage{amsthm}
\usepackage{amssymb}
\usepackage{latexsym,array}
\usepackage{amsfonts}
\usepackage{shadow}
\usepackage{amsbsy}
\usepackage{amssymb}
\usepackage{dsfont}
\usepackage{mathtools}
\usepackage{enumitem}
\usepackage{doi}

\usepackage[notref, notcite, final]{showkeys}
\usepackage{color}
\usepackage{graphicx}
\usepackage[hidelinks]{hyperref}
\usepackage{cleveref}
\usepackage{fullpage}
\usepackage{comment}
\usepackage{tikz-cd}
\usepackage{tikz}
\usepackage{float}

\usepackage{dsfont}

\usepackage{cleveref}

\numberwithin{equation}{section}

\newcommand{\bigD}{\mathcal{D}}

\newtheorem{theorem}{Theorem}[section]
\newtheorem{lemma}[theorem]{Lemma}
\newtheorem{corollary}[theorem]{Corollary}
\newtheorem{proposition}[theorem]{Proposition}

\theoremstyle{definition}
\newtheorem{remark}[theorem]{{\bf Remark}}
\newtheorem{definition}[theorem]{Definition}

\newcommand{\cc}{\mathbb{C}}

\newcommand{\hh}{\mathbb{H}}

\newcommand{\pp}{\partial}
\newcommand{\rr}{\mathbb{R}}

\crefname{enumi}{}{}
\crefname{enumii}{}{}

\title[]{A polyanalytic functional calculus \\of order 2 on the $S$-spectrum}
\author{ANTONINO DE MARTINO} 
\address{Politecnico di Milano, Dipartimento di Matematica, Via E. Bonardi, 9 20133 Milano, Italy.}
\email{antonino.demartino@polimi.it}

\author{STEFANO PINTON}
\address{Politecnico di Milano, Dipartimento di Matematica, Via E. Bonardi, 9 20133 Milano, Italy.}
\email{stefano.pinton@polimi.it}


\begin{document}
	
	\maketitle
	
\begin{abstract}
The Fueter theorem provides a two step procedure to build an axially monogenic function, i.e. a null-solutions of the Cauchy-Riemann operator in $ \mathbb{R}^4$, denoted by $ \mathcal{D}$. In the first step a holomorphic function is extended to a slice hyperholomorphic function, by means of the so-called slice operator. In the second step a monogenic function is built by applying the Laplace operator in four real variables ($\Delta$) to the slice hyperholomorphic function. In this paper we use the factorization of the Laplace operator, i.e.  $\Delta= \mathcal{\overline{D}} \mathcal{D}$ to split the previous procedure. From this splitting we get a class of functions that lies between the set of slice hyperholomorphic functions and the set of axially monogenic functions: the set of axially polyanalytic functions of order 2, i.e. null-solutions of $ \mathcal{D}^2$. We show an integral representation formula for this kind of functions. The formula obtained is fundamental to define the associated functional calculus on the $S$-spectrum. As far as the authors know, this is the first time that a monogenic polyanalytic functional calculus has been taken into consideration.   
\end{abstract}

\medskip

\section{Introduction}
The Fueter theorem provides a procedure to extend holomorphic functions to quaternionic-valued functions in the kernel of the operator $ \mathcal{D}$, which is a suitable generalization of the Cauchy-Riemann operator in $  \mathbb{R}^4$. The Fueter theorem is performed in two steps. In the first step the slice operator $T_F$ is applied to $ \mathcal{O}(D)$, which is the set of holomorphic functions on $ D \subseteq \mathbb{C}$, in order to get the set of slice hyperholomorphic functions on $ \Omega_D \subset \mathbb{R}^4$, which  is an open set induced by $D$ (see Theorem \ref{FS} in the sequel). This set of functions is denoted by $ \mathcal{SH}(\Omega_D)$. In the second step the Laplace operator in four real variables is applied to $ \mathcal{SH}(\Omega_D)$ to get the set of axially monogenic functions, denoted by $ \mathcal{AM}(\Omega_D)$. It is possible to visualize the previous construction by the following diagram
\begin{equation}
\label{diag1}
\mathcal{O}(D)\overset{T_{F}}{\longrightarrow} \mathcal{SH}(\Omega_D)\overset{\Delta}{\longrightarrow} \mathcal{AM}(\Omega_D).
\end{equation}
From the Fueter mapping theorem it is possible to deduce two spectral theories. From the first step one can deduce the spectral theory on the $S$-spectrum associated with the Cauchy formula of slice hyperholomorphic functions. By applying the Laplacian of four real variables to the slice hyperholomorphic Cauchy kernel it is possible to obtain a Fueter theorem in integral form and from this it is possible to deduce a monogenic functional calculus, called $ \mathcal{F}$- functional calculus. This procedure is illustrated in the following diagram
{\small
\begin{equation}
\label{diag2}
	\begin{CD}
		{\mathcal{SH}(U)} @.  {\mathcal{AM}(U)} \\   @V  VV
		@.
		\\
		{{\rm  Slice\ Cauchy \ Formula}}  @> \Delta>> {{\rm Fueter\ theorem \ in \  integral\  form}}
		\\
		@V VV    @V VV
		\\
		{{\rm  S-Functional \ calculus}} @. {{\rm \mathcal{F}-functional \ calculus}}
	\end{CD}
\end{equation}
}

In \cite{CDPS} the authors have studied a  possible splitting of the diagram \eqref{diag2} and have showed that between the set of slice hyperholomorphic functions and the set of axially monogenic functions lies the set of of axially harmonic functions. Moreover, by means of this splitting, they developed an harmonic functional calculus.  
\\ The main goal of this paper is to understand another, different splitting of \eqref{diag2}.

By rearranging the maps in the Fueter theorem it is possible to get the set of axially polyanalytic functions of order 2, i.e. functions in the kernel of $ \mathcal{D}^2$, see Section 2.  In this paper we study the splitting
\begin{equation}
\mathcal{O}(D) \overset{T_{F}}{\longrightarrow} \mathcal{SH}(\Omega_D)\overset{\mathcal{\overline{D}}}{\longrightarrow} \mathcal{AP}_2(\Omega_D)\overset{\mathcal{D}}{\longrightarrow}\mathcal{AM}(\Omega_D),
\end{equation} 
where $ \mathcal{AP}_2(\Omega_D)$ is the set of axially polyanalytic functions of order 2. The goal of this paper is to describe the central part of this diagram
\newline
\newline
{\small \small
	\begin{equation}
\label{diagg}
		\begin{CD}
			{\mathcal{SH}(U)} @. {\mathcal{AP}_2(U)}  @.  {\mathcal{AM}(U)} \\   @V  VV
			@.
			\\
			{{\rm  Slice\ Cauchy \ Formula}}  @> \overline{\mathcal{D}}>> {{\rm \mathcal{AP}_2 \ integral\ form}}@> \mathcal{D}  >> {{\rm Fueter\ thm. \ in \  integral\  form}}
			\\
			@V VV    @V VV  @V VV
			\\
			$S$-{{\rm Functional \ calculus}} @. {{\rm Polyanalytic \ functional \ calculus  \ of \ order \ 2}}@. \mathcal{F}-{{\rm functional \ calculus}}
		\end{CD}
	\end{equation}
}
\\ Axially polyanalytic functions play an important role in the study of elasticity problems, see \cite{K,M}. The theory of polyanalytic functions is also used to investigate problems in time-frequency analysis, see for example \cite{A,DMD2} and to study well-known Hilbert spaces, see for instance \cite{A1, B1, V}. For more information about polyanalytic functions see \cite{AF, B1976}.

In order to clarify the outline and contents of the paper we need to fix the notations. Let $\mathbb H$ be the skew field of quaternions
$$ \mathbb{H}:= \{q=q_0+q_1e_1+q_2e_2+q_3e_3 \, | \, q_0, q_1, q_2, q_3 \in \mathbb{R}\},$$
where the imaginary units satisfy the following relations
$$e_1^2=e_2^2=e_3^2=-1\quad \text{and}\quad e_1e_2=-e_2e_1=e_3,\ \  e_2e_3=-e_3e_2=e_1, \ \  e_3e_1=-e_1e_3=e_2.$$
Let $q \in \mathbb{H}$. We call $ \hbox{Re}(q):=q_0$ the real part of $q$ and $ \underline{q}= q_1e_1+q_2e_2+q_3e_3$ the imaginary part. We define the conjugate of $ q \in \mathbb{H}$ as $ \bar{q}=q_0- \underline{q}$ and the modulus of $ q $ as $|q|= \sqrt{q \bar{q}}$. We define $\mathbb S:=\{q\in\hh:\, \operatorname{Re}(q)=0\,\textrm{and}\, |q|=1\}$. We observe that if $J \in \mathbb{S}$ then $J^2=-1$. Therefore we can consider $J$ as an imaginary unit, and we denote by $\cc_J:=\{u+Jv\in\hh:\, u,v\in\rr\}$, an isomorphic copy of the complex numbers.
\\ We recall that the Fueter operator $ \mathcal{D}$ and its conjugate $ \mathcal{\overline{D}}$ are defined as follows
$$ \mathcal{D}:= \partial_{q_0}+ \sum_{i=1}^3 e_i \partial_{q_i} \quad \hbox{and} \quad \mathcal{\overline{D}}:= \partial_{q_0}- \sum_{i=1}^3 e_i \partial_{q_i}.$$
In this paper we show that it is possible to have an integral representation of axially polyanalytic functions of order 2 in terms of slice hyperholomorphic functions, see Section 4. More precisely, let $W\subset \hh$ be an open set and $U$ be a slice Cauchy domain such that $\overline U\subset W$. Then for $J\in\mathbb S$ and $ds_J=ds(-J)$ we have that if $f$ is left slice hyperholomorphic on $W$, then the function $\breve f^0(q)=\overline\bigD f(q)$ is polyanalytic of order $2$ and it admits the following integral representation  
$$ \breve f^0(q)=-\frac 1{2\pi}\sum_{k=0}^1(-q_0)^k\int_{\partial(U\cap\cc_J)} F_L(s,q) s^{1-k} \, ds_J\, f(s)\quad \forall q\in U; $$
where $F_L(s,q):=-4(s-\bar q)(s^2-2\operatorname{Re}(q)s+|q|^2)^{-2}$. A similar integral formula is obtained for right slice hyperholomorphic functions.
\\Furthermore, we show that it is possible to expand in series the kernel of the integral representation, see Section 5. We prove that the series is written in terms of the so-called Clifford-Appell polynomials, see \cite{CMF1}.
\\ All these tools are extremely good to define a monogenic polyanalytic functional calculus on the $S$-spectrum for bounded operators with commuting components, see Section 6. To be more precise let $T=T_0+T_1e_1+T_2e_2+T_3e_3$ be a quaternionic bounded linear operator with commuting components $T_i$, $i=0,...,3$. By considering the following operator
$$ \mathcal Q_{c,s}(T):=s^2-2\operatorname{Re}(T)s+T\bar T, $$ 
where $ \bar{T}:= T_0- e_1T_1-e_2T_2-e_3T_3$ we define the commutative $S$-spectrum of $T$ as
$$ \sigma_S(T):=\{s\in\hh :\, \mathcal{Q}_{c,s}(T)^{-1} \quad  \hbox{is not invertible}\} $$
and the $S$-resolvent set of $T$ as
$$\sigma_S(T):=\hh\setminus \sigma_S(T).$$
It turns out that the commutative $S$-spectrum is the same of the $S$-spectrum defined in \cite{CGKBOOK,CSS4} when we deal with operators with commuting components.
\\Roughly speaking for every function $\breve f^0:=\overline\bigD f$, with $f$ a left slice hyperholomorphic function, we define the polyanalytic functional calculus of order 2 as 
$$\breve f^0(T)=\frac 1{2\pi}\int_{\partial (U\cap\cc_J)} \mathcal{P}_2^L(s,T)\, ds_J\, f(s),$$
where $\mathcal{P}_2^L(s,T)=\sum_{j=0}^1 T_0^j(-1)^{j+1}F_{L}(s, T)s^{1-j}$, with $F_L(s,T):=-4 (s-\bar T)\mathcal Q_{c,s}(T)^{-2}$. Moreover, $U$ is an arbitrary bounded slice Cauchy domain with $\sigma_S(T) \subset U$, $ \bar{U} \subset \hbox{dom}(f)$, $ds_J=ds(-J)$ and $J \in \mathbb{S}$ is an arbitrary imaginary unit. A similar definition is valid for $\breve f^0= f\overline\bigD$ with $f$ a right slice hyperholomorphic function. 
Recently, a polyanalytic functional calculus has been developed in \cite{ACDS}, but for slice hyperholomorphic polyanalytic functions. To the best of our knowledge, this is the first time that a monogenic polyanalytic functional calculus has been considered.

\section{Preliminaries}

We recall some basic results and notions that we need in the sequel (for a complete introduction to this topics we refer to the book \cite{CGKBOOK}). 
We say that $U\subset \mathbb H$ is axially symmetric if, for every $u+Iv\in U$, all the elements $u+Jv\in U$ for every $J\in\mathbb S$. Moreover, we say that $U$ is a slice domain if $U\cap\rr\neq 0$ and $U\cap \cc_J$ is a domain in $\cc_J$ for every $J\in\mathbb S$. Let $U\subset\mathbb H$ be an open axially symmetric set. Let $\mathcal U\subset \rr\times\rr$ be such that $q=u+Jv\in U$ for all $(u,v)\in\mathcal U$. We say that a function $f: U\to \hh$ of the form 
\begin{equation}
\label{type}
f(q)=\alpha(u,v)+J\beta(u,v)
\end{equation}
is left slice hyperholomorphic if $\alpha$ and $\beta$ are $\hh$-valued differentiable functions such that 
$$ \alpha(u,v)=\alpha(u,-v) \quad\textrm{and}\quad \beta(u,v)=-\beta(u,-v)\quad\textrm{for any $(u,v)\in\mathcal U$}.  $$ 
Moreover, $\alpha,\, \beta$ satisfy the Cauchy-Riemann system
$$\partial_u \alpha(u,v)-\partial_v\beta(u,v)=0\quad\textrm{and}\quad\partial_v\alpha(u,v)+\partial_u\beta(u,v)=0.$$
We recall that right slice hyperholomorphic functions are of the form
\begin{equation}
\label{type2}
f(q)=\alpha(u,v)+\beta(u,v)J,
\end{equation}
where $\alpha,\, \beta$ satisfy the above conditions. The set of left (resp. right) slice hyperholomorphic functions on $U$ is denoted by $\mathcal{SH}_L(U)$ (resp. $\mathcal{SH}_R(U)$). If we do not care about the right and the left we denote this set of functions as $ \mathcal{SH}(U)$.
\\ Moreover, we say that a function is left (resp. right) axially monogenic if it is of the form \eqref{type} (resp. of the form \eqref{type2})
and it is in the kernel of the Fueter operator $ \mathcal{D}$ i.e. $\mathcal{D}f=0$ (resp. $f\mathcal{D}=0$). We denote the set of these functions (no matter left or right) as $ \mathcal{AM}(U)$.

If $U$ is an axially symmetric Cauchy domain and $\overline U\subset W$ for some open axially symmetric domain $W\subset \hh$, then for any $f\in \mathcal{SH}_L(W)$ (resp. $f\in \mathcal{SH}_R(W)$) we have the so called Cauchy formulas
\begin{equation}
\label{cauchy}
f(q)=\frac 1{2\pi}\int_{\partial(U\cap\cc_J)} S^{-1}_L(s,q)\, ds_J\, f(s)\quad \left(\textrm{resp. } \frac 1{2\pi}\int_{\partial(U\cap\cc_J)} f(s) \, ds_J\, S^{-1}_R(s,q)\right),
\end{equation}
where $J\in\mathbb S$, $ds_J=ds(-J)$,
$$ S^{-1}_L(s,q)=(s-\bar q)(s^2-2\operatorname{Re}(q)s+|q|^2)^{-1}\quad \textrm{and}\quad S^{-1}_R(s,q)=(s^2-2\operatorname{Re}(q)s+|q|^2)^{-1} (s-\bar q),$$
see \cite[Thm 2.1.32]{CGKBOOK}. The function $S^{-1}_L(s,q)$ (resp. $S^{-1}_R(s,q)$) is called the left (resp. right) slice hyperholomorphic Cauchy kernel. 
\\The definition of slice hyperholomorphic function, that we adopt in this paper is the most appropriate one for the operator theory and it comes from the Fueter mapping theorem (see \cite{CSS2, F}). This theorem puts in relation the slice hyperholomorphicity and the axially monogenicity.

In \cite{CSS}, \cite[Theorem $2.2.6$]{CGKBOOK} it is  proved that a Fueter theorem in integral form. The main advantages of this method is that one can obtain a monogenic function by integrating the suitable slice hyperholomorphic functions.
Let us consider $f\in \mathcal{SH}_L(W)$ (resp. $f\in \mathcal{SH}_R(W)$) we can state the Fueter theorem in integral form in the following way
$$\breve{f}(q)= \Delta f(q)=\frac 1{2\pi}\int_{\partial (U\cap\cc_J)} F_L(s,q)\, ds_I\, f(s)\quad \left(\textrm{resp. }  \breve{f}(q)=\Delta f(q)=\frac 1{2\pi}\int_{\partial (U\cap\cc_J)} f(s)\, ds_I\, F_R(s,q) \right)$$
where 
$$F_L(s,q):=\Delta S^{-1}_L(s,q)=-4(s-\bar q)(s^2-2\operatorname{Re}(q)s+|q|^2)^{-2},$$
and 
$$F_R(s,q):=\Delta S^{-1}_R(s,q)=-4(s^2-2\operatorname{Re}(q)s+|q|^2)^{-2} (s-\bar q)$$
are respectively the left and the right $\mathcal{F}$-kernels, then the function $\breve{f}$ is left (resp. right) axially  monogenic. This is due fact that the left (resp. right) $\mathcal{F}$-kernel is a left (resp. right) axially monogenic function in the variable $q$ and right (resp. left) slice hyperholomorphic function in the variable $s$ (see \cite[Theorem 2.2.2]{CGKBOOK} and \cite[Prop. 2.2.4]{CGKBOOK}). 

Now we want to introduce the $SC$-functional calculus, see \cite{CS}. This functional calculus is the commutative version of the $S$- functional calculus, see \cite{CGKBOOK, CSS4}. Let $X$ be a two sided quaternionic Banach module of the form $X= X_{\mathbb{R}} \otimes \mathbb{H}$, where $X_{\mathbb{R}}$ is a real Banach space. In this paper we consider $\mathcal{B}(X)$ the Banach space of all bounded right linear operators acting on $X$. In the sequel we will consider bounded operators of the form $T=T_0+T_1e_1+T_2e_2+T_3e_3$, with commuting components $T_{i}$ acting on a real vector space $X_{\mathbb{R}}$, i.e., $T_i \in \mathcal{B}(X_{\mathbb{R}})$ for $i=0,1,2,3$. The subset of $ \mathcal{B}(X)$ given by the operators $T$ with commuting components $T_i$ is denoted by $ \mathcal{BC}(X)$. 
Now, we define the $SC$-functional calculus. Let $T \in \mathcal{BC}(X)$ and  let $U$ be a slice Cauchy domain such that $\sigma_F(T)\subset U$  and $\overline{U}\subset W\subset\mathbb H$ where $W$ is an axially symmetric open domain. We define for every $f\in \mathcal{SH}_L(W)$ (resp. $f\in \mathcal{SH}_R(W)$ )
\begin{equation}
	\label{Scalleft}
	f(T):={{1}\over{2\pi }} \int_{\partial (U\cap \mathbb{C}_J)} S_L^{-1} (s,T)\  ds_J \ f(s) \qquad \left( \hbox{resp} \quad	f(T):={{1}\over{2\pi }} \int_{\partial (U\cap \mathbb{C}_J)} \  f(s)\ ds_J
	\ S_R^{-1} (s,T)\right),
\end{equation}
where 
\begin{equation}
S^{-1}_L(T):=(s-\bar T)\mathcal Q_{c,s}(T)^{-1} \quad \hbox{and} \quad S^{-1}_R(T):=\mathcal Q_{c,s}(T)^{-1} (s-\bar T)
\end{equation}
are called respectively the left and right $S$-resolvent operators.
The definition of $SC$-functional calculus is well posed since the integrals in (\ref{Scalleft}) depend neither on $U$ and nor on the imaginary unit $J\in\mathbb{S}$, see \cite[Thm. 3.2.6]{CGKBOOK}. 
\begin{remark}
It is possible to define the $S$-functional calculus for $T \in \mathcal{B}(X)$, see \cite{CGKBOOK, CSS4}. It works also for fully Clifford operators with non commuting components, see \cite{CKPS}. However, for our purpose it is enough to consider the operator in $ \mathcal{BC}(X)$.  
\end{remark}
Now we define the $\mathcal{F}$-functional calculus. It is defined on the $ S$-spectrum but it generates a monogenic functional calculus in the spirit of McIntosh and collaborators, see \cite{J, JM}. Let $T=T_0+ T_1e_1+T_2e_2+T_3e_3 \in\mathcal{BC}(X)$, assume that the operators $T_{\ell}$, $\ell=0,1,2,3$ have real spectrum and set $ds_J=ds/J$, where $J\in \mathbb{S}$.
For $ \breve{f}=\Delta f$ with $f\in \mathcal{SH}_L(W)$ (resp. $f\in \mathcal{SH}_R(W)$), we define
\begin{equation}\label{DefFCLUb}
	\breve{f}(T):=\frac{1}{2\pi}\int_{\pp(U\cap \mathbb{C}_J)} F_L(s,T) \, ds_J\, f(s) \qquad \left( \hbox{resp.} 	\breve{f}(T):=\frac{1}{2\pi}\int_{\pp(U\cap \mathbb{C}_J)} f(s) \, ds_J\, F_R(s,T)\right),
\end{equation}
where 
\begin{equation}
\label{fres}
F_L(s,T):=-4 (s-\bar T)\mathcal Q_{c,s}(T)^{-2} \quad \hbox{and} \quad F_R(s,T):=-4 \mathcal Q_{c,s}(T)^{-2} (s-\bar T)
\end{equation}
are called respectively  the left and the right $\mathcal{F}$-resolvent operators.
The definition of the $\mathcal{F}$-functional calculus is well posed, see \cite[Thm. 7.1.12]{CGKBOOK}.

\section{Function spaces of axial type in the quaternionic setting}
In \cite{CDPS} the authors gave the following

\begin{definition}[Fine structure of slice hyperholomorphic spectral theory]
A fine structure of a slice hyperholomorphic spectral theory is the set of functions spaces and the associated functional calculi induced by a factorization of the operator $ \Delta$.
\end{definition}
In the quaternionic case only two fine structures are possible. One of them is studied in \cite{CDPS}, and the other one is the main topic of this paper.  
\\ The first fine structure studied corresponds to the factorization $ \Delta= \mathcal{D} \mathcal{\overline{D}}$. In that case, we have the following diagram   
\begin{equation}
\label{fine1}
\mathcal{O}(D) \overset{T_{F}}{\longrightarrow} \mathcal{SH}(\Omega_D)\overset{\mathcal{D}}{\longrightarrow} \mathcal{AH}(\Omega_D)\overset{\mathcal{\overline{D}}}{\longrightarrow}\mathcal{AM}(\Omega_D),
\end{equation}
where the $ \mathcal{AH}(\Omega_D)$ is the set of axially harmonic functions and $ \Omega_D$ is defined as in Theorem \ref{FS}.
\\ The aim of this paper is to study the fine structure which corresponds to the other possible factorization of the Laplacian, $\Delta= \mathcal{\overline{D}} \mathcal{D}$. To this end, we need the following splitting of the Fueter theorem (see \cite{F}).
\begin{theorem}
\label{FS}
Let $f_{0}(z)= \alpha(u,v)+i \beta(u,v)$ be a holomorphic function defined in a domain (open and connected) $D$ in the upper-half complex plane and let
$$ \Omega_D=\{q=q_0+q_{1}e_1+q_{2}e_{2}+q_3 e_{3} \,\  |\  \ (q_0, |\underline{q}|) \in D\}$$
be the open set induced by $D$ in $\mathbb{H}$. The map
$$ f(q)= T_{F}(f_0):= \alpha(q_0, |\underline{q}|)+ \frac{\underline{q}}{|\underline{q}|}\beta(q_0, |\underline{q}|)$$
takes the holomorphic function $f_0(z)$ and gives the intrinsic slice hyperholomorphic function $f$ induced by $f_0$. Then the function
$$ \breve f^0(q):= \mathcal{\overline{D}} \left(\alpha(q_0, | \underline{q}|)+ \frac{\underline{q}}{| \underline{q}|} \beta(q_0, | \underline{q}|)\right),$$
is in the kernel of $ \mathcal{D}^2$, i.e.
$$ \mathcal{D}^2 \breve{f}^0=0 \qquad \hbox{on} \quad \Omega_D.$$
Moreover,
$$ \breve{f}(q)= \mathcal{D} \breve{f}^0(q),$$
is axially monogenic.
\end{theorem} 
From the previous theorem we have the following diagram
\begin{equation}
\label{fine2}
\mathcal{O}(D) \overset{T_{F}}{\longrightarrow} \mathcal{SH}(\Omega_D)\overset{\mathcal{\overline{D}}}{\longrightarrow} \mathcal{AP}_2(\Omega_D)\overset{\mathcal{D}}{\longrightarrow}\mathcal{AM}(\Omega_D),
\end{equation}
where $\mathcal{AP}_2(\Omega_D)$ is the set of axially polyanalytic functions of order 2. Now,  we give a rigorous definition of this set
\begin{definition}[Axially polyanalytic function of order 2]
Let $U \subseteq \mathbb{H}$ be an axially symmetric open set not intersecting the real line, and let
$$
\mathcal{U} =\{ (u,v)\in \mathbb{R}\times\mathbb{R}^+ \ | \ u+\mathbb{S}v \in U\}.
$$
Let $f: U \to \mathbb{H}$ be a function, of class $\mathcal{C}^3$, of the form
$$ f(q)=\alpha(u,v)+J\beta(u,v), \quad q=u+Jv, \quad J \in \mathbb{S},$$
where $\alpha$ and $\beta$ are quaternionic-valued functions.
More in general, let $U \subseteq \mathbb{H}$ be an axially symmetric open set and let
$$
\mathcal{U} =\{ (u,v)\in \mathbb{R}^2 \ | \ u+\mathbb{S}v \in U\},
$$
and assume that
\begin{equation}\label{eveoddll}
	\alpha(u,v)=\alpha(u,-v), \ \ \ \  \beta(u,v)=-\beta(u,-v) \ \\ \ \  \hbox{for all} \, \, (u,v) \in \mathcal{U}.
\end{equation}
Let us set

$$ \breve{f}^0(q):= \mathcal{\overline{D}} f(q) \qquad \hbox{for} \quad q \in U.$$
If 
$$ \mathcal{D}^2 \breve{f}^0(q)=0, \quad \hbox{for} \quad q \in U,$$
we say that $\breve{f}^0$ is axially polyanalytic of order 2.

\end{definition}
It is possible to write a polyanalytic function as a sum of axially monogenic functions, see \cite{B1976}. Specifically, we can write the so called polyanalytic decomposition as
\begin{equation}
\label{In1}
\breve{f}^0(q)=  \breve{f}_0(q)+q_0 \breve{f}_1(q),
\end{equation}
where the $ \breve{f}_0(q)$ and $ \breve{f}_1(q)$ are axially monogenic functions.  
\\As well as the monogenic functions satisfy a system of differential equations, called Vekua systems \cite{CS0} also the axially polyanalytic functions of order 2 satisfy a system of differential equations, but of order two. The following result will be investigated in a forthcoming work.

\begin{theorem}
Let $U$ be an axially symmetric open set in $ \mathbb{H}$, not intersecting the real line, and let $ \breve{f}^0(q)= A(q_0,r)+ \underline{\omega} B(q_0,r)$ be an axially polyanalytic function of order 2 on $U$, $r>0$ and $\underline{\omega} \in \mathbb{S}$. Then the functions $A=A(q_0,r)$ and $B=B(q_0,r)$ satisfy the following system
$$
\begin{cases}
\partial_{x_0}^2 A-2 \partial_{x_0} \partial_r B- \frac{4}{r} \partial_{x_0} B- \partial_{r}^2 A- \frac{2}{r} \partial_r A=0\\
\partial_{x_0}^2 B+2 \partial_{x_0}\partial_{r} A- \partial_r^2B-2 \frac{r\partial_r B-B}{r^2}=0.
\end{cases} 
$$
\end{theorem}  
\begin{remark}
A similar system of the same order holds for the axially harmonic functions, see \cite{CDPS}.
\end{remark}
In conclusion, even if $ \Delta= \mathcal{D} \mathcal{\overline{D}}= \mathcal{\overline{D}} \mathcal{D}$ the application of $\mathcal{D}$ or the operator $ \mathcal{\overline{D}}$ to the set of slice hyperholomorphic functions gives arise to two completely different fine structures.

\section{Integral representation of polyanalytic functions of order 2}
In this section we  show how to write a polyanalytic function of order 2  in integral form. The main advantage of this approach is that it is enough to compute an integral of slice hyperholomorphic functions  in order to get a polyanalytic function of order 2. The crucial point to show the integral representation is to apply the operator $ \mathcal{\overline{D}}$ to the slice hyperholomorphic Cauchy kernels.

\begin{theorem}
\label{res2}
Let $s$, $q \in \mathbb{H}$, be such that $s \notin [q]$ then
\begin{equation}
\label{leftdeco}
\mathcal{\overline{D}}S^{-1}_L(s,q)=- F_{L}(s,q)s+ q_0 F_{L}(s,q)= \sum_{k=0}^{1} q_0^k F_L(s,q) (-1)^{k+1} s^{1-k},
\end{equation}
and
\begin{equation}
\label{rightdeco}
S^{-1}_R(s,q)\mathcal{\overline{D}}=- sF_{R}(s,q)+ q_0 F_{R}(s,q)= \sum_{k=0}^{1} q_0^k s^{1-k}F_R(s,q) (-1)^{k+1} .
\end{equation}
\end{theorem}
\begin{proof}
We start by applying the derivative with respect to $ \partial_{q_0}$ to the left slice hyperholomorphic Cauchy kernel
\begin{equation}
\label{comp1}
\partial_{q_0} S^{-1}_L(s,q)=- \mathcal{Q}_{c,s}(q)^{-1}+ \frac{q_0}{2} F_L(s,q)- \frac{1}{2} F_L(s,q)s.
\end{equation}
Now, we make the derivative with respect to $ \partial_{q_i}$,
\begin{equation}
\label{comp2}
\partial_{q_i} S^{-1}_L(s,q)= e_i \mathcal{Q}_{c,s}(q)^{-1}+ \frac{q_i}{2} F_L(s,q), \qquad i=1, 2,3. 
\end{equation}
Formula \eqref{comp1} and \eqref{comp2} imply that
\begin{eqnarray*}
\mathcal{\overline{D}} S^{-1}_L(s,q)&=&\left(\partial_{q_0}- \sum_{i=1}^3 e_{i} \partial_{q_i}\right)S^{-1}_L(s,q)=2\mathcal{Q}_{c,s}(q)^{-1}+ \frac{q_0}{2} F_L(s,q)- \frac{1}{2} F_L(s,q)s- \frac{\underline{q}}{2} F_L(s,q)\\
&=& 2 \mathcal{Q}_{c,s}(q)^{-1}+ \frac{\bar{q}}{2} F_L(s,q)- \frac{1}{2} F_L(s,q) s.
\end{eqnarray*}
From the equality $F_{L}(s,q)s-q F_L(s,q)=-4 \mathcal{Q}_{c,s}(q)^{-1}$ it follows the thesis.
By similar computations we obtain formula \eqref{rightdeco}.
\end{proof}
Now, we study the regularities of $ \mathcal{\overline D} S_{L}^{-1}(s,q)$ and $S^{-1}_R(s,q)\mathcal{\overline{D}}$ both in $s$ and in $q$.
\begin{proposition}
\label{sh}
Let $s$, $q \in \mathbb{H}$, be such that $s \notin [q]$. The function $ \mathcal{\overline{D}}S^{-1}_L(s,q)$ is a right slice hyperholomorphic function in the variable $s$, while $S^{-1}_R(s,q)\mathcal{\overline{D}}$ is left slice hyperholomorphic in the variable $s$.
\end{proposition}
\begin{proof}
By Theorem \ref{res2} we know that $ \mathcal{\overline{D}}S^{-1}_{L}(s,q)$ is a sum of right slice hyperholomorphic functions in the variable $s$. Indeed $ \mathcal{Q}_{c,s}(q)^{-1}$ is a right slice hyperholomorphic function as well as $ \bar{q} F_L(s,q)$ and $F_L(s,q)s$. The left slice hyperholomorphicity of the function $S^{-1}_R(s,T)\mathcal{\overline{D}}$ follows by similar arguments.
\end{proof}

\begin{proposition}
\label{regu2}
Let $s$, $q \in \mathbb{H}$, be such that $s \notin [q]$. The function $ \mathcal{D}S^{-1}_L(s,q)$ is left polyanalytic of order 2 and $S^{-1}_R(s,T)\mathcal{\overline{D}}$ is right polyanalytic of order 2.
\end{proposition}
\begin{proof}
It follows from the fact that the function $ F_L(s,q)$ is axially monogenic in the variable $q$ and the Laplace operator is a real operator, thus it can commute with other operators. Therefore, we get
$$ \mathcal{D}^2 \left(\mathcal{\overline{D}} S^{-1}_L(s,q)\right)= \mathcal{D} \Delta S^{-1}_L(s,q)= \mathcal{D} F_L(s,q)=0.$$
The right polyanalyticity of $S^{-1}_R(s,T)\mathcal{\overline{D}}$ follows similarly.
\end{proof}

The expressions obtained in Theorem \ref{res2} can be considered a polyanalytic decomposition of $ \mathcal{\overline{D}}S^{-1}_L(s,q)$ and $ S^{-1}_R(s,q)\mathcal{\overline{D}}$, respectively, see formula \eqref{In1}. Indeed the functions $-F_L(s,q)s$ and $F_L(s,q)$ are left axially monogenic in the variable $q$. Similarly, the functions $-sF_R(s,q)$ and $F_R(s,q)$ are right axially monogenic in the variable $q$.
\newline
\newline
Now, we have all what we need to write an axially polyanalytic function of order 2 as an integral formula. This will be fundamental to define the polyanalytic functional calculus of order 2 based on the $S$-spectrum.
\begin{theorem}[Integral representation of axially polyanalytic functions of order $2$]
\label{inrap}
Let $W\subset \hh$ be an open set. Let $U$ be a slice Cauchy domain such that $\overline U\subset W$. Then for $J\in\mathbb S$ and $ds_J=ds(-J)$ we have 
\begin{enumerate}
\item if $f\in \mathcal{SH}_L(W)$, then the function $\breve f^0(q)=\overline\bigD f(q)$ is polyanalytic of order $2$ and it admits the following integral representation 
\begin{equation}
\label{star}
\breve f^0(q)=-\frac 1{2\pi}\sum_{k=0}^1(-q_0)^k\int_{\partial(U\cap\cc_J)} F_L(s,q) s^{1-k} \, ds_J\, f(s)\quad \forall q\in U;
\end{equation}
\item if $f\in \mathcal{SH}_R(W)$, then the function $\breve f^0(q)= f(q) \overline\bigD$ is polyanalytic of order $2$ and it admits the following integral representation  
\begin{equation}
\label{start2}
\breve f^0(q)=-\frac 1{2\pi}\sum_{k=0}^1(-q_0)^k\int_{\partial(U\cap\cc_J)} f(s) \, ds_J\,  s^{1-k} F_R(s,q)\quad \forall q\in U.
\end{equation}
\end{enumerate}
The integrals depend neither on $U$ nor on the imaginary unit $J\in U$.
\end{theorem}
\begin{proof}
We get the thesis by applying the conjugate Fueter operator $ \mathcal{\overline{D}}$ to the Cauchy formulas, see \eqref{cauchy}. By Theorem \ref{res2} it follows \eqref{star} and \eqref{start2}. Finally, the function $\breve f^0(q)$ is polyanalytic of order 2 by Proposition \ref{regu2}.
\end{proof}
In this section we have described the second central row of the diagram \eqref{diagg}. From this section is clear the reason of the lack of the arrow that connects the set of axially polyanalytic functions and their integral representation. Indeed, we obtain it by means of the slice Cauchy formula.

\section{Series expansion of the kernel of the fine structure spaces}
In this section our aim is to address the following 
\newline
\newline
\textbf{Problem }
Is it possible to write a series expansion of $ \mathcal{\overline{D}}S^{-1}_L(s,q)$ and $ S^{-1}_R(s,q) \mathcal{\overline{D}}$ in terms of $q$ and $ \bar{q}$?
\newline
\newline
In order to answer this question we need the following series expansion of the slice hyperholomorphic Cauchy kernels, see \cite[Thm. 2.1.22]{CGKBOOK} \cite{CSS4}. For $q$, $s \in \mathbb{H}$ with $|q| < |s|$ we have
\begin{equation}
\label{expa}
S^{-1}_L(s,q)=\sum_{n=0}^\infty q^n s^{-1-n}, \qquad S^{-1}_R(s,q)=\sum_{n=0}^\infty  s^{-1-n}q^n . 
\end{equation}
Therefore it is clear that in order to find the series expansions of $ \mathcal{\overline{D}}S^{-1}_L(s,q)$ and $ S^{-1}_R(s,q) \mathcal{\overline{D}}$ it is fundamental to understand the action of the conjugate Fueter operator $ \mathcal{\overline{D}}$ over the monomial $q^n$. 
\begin{lemma}
\label{res3}
For $n \geq 1$ we have
\begin{equation}
\label{F1}
\mathcal{\overline{D}} q^n= 2 \left( n q^{n-1}+ \sum_{k=1}^n q^{n-k} \bar{q}^{k-1} \right).
\end{equation}
Moreover,
\begin{equation}
\label{comm}
q^n\mathcal{\overline{D}}=\mathcal{\overline{D}} q^n.
\end{equation}
\end{lemma}
\begin{proof}
By \cite[Lemma 1]{B} we know that
$$\mathcal{D} q^n=(\partial_{q_0}+ \mathcal{D}_{\underline{q}})q^n=-2 \sum_{k=1}^n q^{n-k} \bar{q}^{k-1},$$
where $ \mathcal{D}_{\underline{q}}:= \sum_{i=1}^3 e_i \partial_{q_i}$.
Then we have
\begin{equation}
\mathcal{D}_{\underline{q}} q^n=-2 \sum_{k=1}^n q^{n-k} \bar{q}^{k-1}-n q^{n-1}.
\end{equation}
Therefore
\begin{equation}
\mathcal{\overline{D}} q^n= \left(\partial_{q_0}-\mathcal{D}_{\underline{q}}\right)q^n=2 \left( n q^{n-1}+ \sum_{k=1}^n q^{n-k} \bar{q}^{k-1} \right).
\end{equation}
Finally formula \eqref{comm} follows with similar computations.
\end{proof}

It is possible to write polynomials $\mathcal{\overline{D}}q^n$ in terms of the Clifford-Appell polynomials in the quaternionic setting, see \cite{CMF1}. This family of axially monogenic homogeneous polynomials is defined as
\begin{equation}
\label{polino}
Q_{\ell}(q,\bar{q})= \frac{2}{(\ell+1) (\ell+2)}\sum_{j=0}^\ell (\ell-j+1) q^{\ell-j} \bar{q}^j, \quad \hbox{for any} \quad \ell \geq 0.
\end{equation}

\begin{proposition}
\label{res4}
Let $n \geq 2$, then for $q \in \mathbb{H}$, we have
\begin{equation}
\label{polydeco1}
\mathcal{\overline{D}} q^n=2 n\sum_{k=0}^1 q_0^k (-1)^{k} (n+1-2k)Q_{n-1-k}(q, \bar{q}).
\end{equation}
\end{proposition}
\begin{proof}
We write
\begin{equation}
\mathcal{\overline{D}} q^n = \left(\mathcal{\overline{D}} q^n-q_{0} \Delta q^n\right)+ q_0 \Delta q^n=g_0(q)+g_1(q),
\end{equation}
and we consider $g_0(q)$. From the fact that
\begin{equation}
\label{lapla}
\Delta q^n =-4 \sum_{k=1}^{n-1}(n-k) q^{n-k-1} \bar{q}^{k-1},
\end{equation}
see \cite[Thm. 3.2]{DKS} and by Lemma \ref{res2} we can write
$$g_0(q)= \mathcal{\overline{D}}q^n-q_0\Delta q^n=2n q^{n-1}+2 \sum_{k=1}^{n} q^{n-k} \bar{q}^{k-1}+4 q_{0} \sum_{k=1}^{n-1}(n-k)q^{n-k-1} \bar{q}^{k-1}.$$
Since $2 q_0=q+ \bar{q}$ we obtain
\begin{eqnarray*}
g_0(q)&=&2 \left(n q^{n-1} + \sum_{k=1}^{n} q^{n-k} \bar{q}^{k-1}+\sum_{k=1}^{n-1}(n-k) q^{n-k} \bar{q}^{k-1}+ \sum_{k=1}^{n-1} (n-k) q^{n-k-1} \bar{q}^k  \right)\\
&=& 2 \left( \sum_{k=1}^{n} q^{n-k} \bar{q}^{k-1}+\sum_{k=1}^{n-1}(n-k) q^{n-k} \bar{q}^{k-1}+ \sum_{k=0}^{n-1} (n-k) q^{n-k-1} \bar{q}^k  \right)\\
&=& 2 \left( \sum_{k=1}^{n} q^{n-k} \bar{q}^{k-1}+\sum_{k=1}^{n}(n-k) q^{n-k} \bar{q}^{k-1}+ \sum_{k=1}^{n} (n-k+1) q^{n-k} \bar{q}^{k-1}  \right)\\
&=& 4 \sum_{k=1}^n (n-k+1) q^{n-k} \bar{q}^{k-1}.
\end{eqnarray*}
By formula \eqref{lapla} we get
$$ g_0(q)=- \Delta(q^{n+1}).$$
This implies that
\begin{equation}
\label{star1}
\mathcal{\overline{D}} q^n=- \Delta(q^{n+1})+ q_0 \Delta q^{n}.
\end{equation}
By \cite[Rem. 3.9]{DKS} we know that for $n \geq 2$ we have
\begin{equation}
\label{malo}
\Delta(q^{n})=-2n(n-1) Q_{n-2}(q, \bar{q})
\end{equation}
where the homogenous polynomials $ Q_n(q, \bar{q})$ are defined in \eqref{polino}. Finally by combining formula \eqref{star1} and formula \eqref{malo} we get
$$
\mathcal{\overline{D}} q^n=2n \left[(n+1) Q_{n-1}(q, \bar{q})-2q_0(n-1) Q_{n-2}(q. \bar{q})\right]=2 n\sum_{k=0}^1 q_0^k (-1)^{k} (n+1-2k)Q_{n-1-k}(q, \bar{q}).
$$
\end{proof}
  
Formula \eqref{polydeco1} can be considered as the polyanalytic decomposition of the polynomials $ \mathcal{\overline{D}} q^n$ since the functions $(n+1) Q_{n-1}(q, \bar{q})$ and $(n-1) Q_{n-2}(q. \bar{q})$ are left and right axially monogenic.

\begin{remark}
The polynomials $ \mathcal{\overline{D}} q^n$ were also obtained in \cite{DMD3}, by means of other tools, see \cite{ADS}.	
\end{remark}

Now, we have all the instruments to introduce the following

\begin{definition}
Let $s$, $q \in \mathbb{H}$, we define the left $ \mathcal{\overline{D}}$-kernel series as
\begin{equation}
\label{F2}
2 \sum_{n=1}^\infty \left( n q^{n-1}+ \sum_{k=1}^{n} q^{n-k} \bar{q}^{k-1} \right) s^{-1-n},
\end{equation}
and the right $ \mathcal{\overline{D}}$-kernel series as
\begin{equation}
\label{F3}
2 \sum_{n=1}^\infty s^{-1-n}\left( n q^{n-1}+ \sum_{k=1}^{n} q^{n-k} \bar{q}^{k-1} \right),
\end{equation}
\end{definition}

\begin{proposition}
\label{res5}
Let $s$, $q \in \mathbb{H}$ with $|q|< |s|$, the left and right $ \mathcal{\overline{D}}$-kernel series are convergent.
\end{proposition}
\begin{proof}
We show only the convergence of the left $ \mathcal{\overline{D}}$-kernel series. The convergence of the right one follows by similar computations.
\\ In order to show the convergence it is enough to prove that the series of moduli is convergent, i.e.
$$ 4  \sum_{n=1}^{+ \infty} n |q|^{n-1} s^{-1-n}.$$
The series converges by the ratio test, indeed
\begin{equation}
\label{conve}
\lim_{n \to + \infty} \frac{(n+1) |q|^n |s|^{-2-n}}{n |q|^{n-1}|s|^{-1-n}}=|q||s|^{-1} <1.
\end{equation}
\end{proof}
The following result contains the solution to the problem stated at the beginning of this section
\begin{lemma}
For $q$, $s \in \mathbb{H}$ such that $|q| < |s|$, we have
\begin{eqnarray*}
\sum_{k=0}^1 q_0^k F_L(s,q)(-1)^{k+1} s^{1-k} &=& 2 \sum_{n=1}^\infty \left( n q^{n-1}+ \sum_{j=1}^{n} q^{n-j} \bar{q}^{j-1} \right) s^{-1-n}  \\
&=& 2 \sum_{n=2}^\infty \sum_{j=0}^1 nq_0^j (-1)^j (n+1-2j) Q_{n-1-j}(q, \bar{q}) s^{-1-n},
\end{eqnarray*}
and
\begin{eqnarray}
\label{P1}
\sum_{k=0}^1 s^{1-k} (-1)^{k+1}F_R(s,q)q_0^k  &=& 2 \sum_{n=1}^\infty s^{-1-n}\left( n q^{n-1}+ \sum_{k=1}^{n} q^{n-k} \bar{q}^{k-1} \right)   \\
\nonumber
&=& 2 \sum_{n=2}^\infty  \sum_{j=0}^1 ns^{-1-n} q_0^j (-1)^j (n+1-2j) Q_{n-1-j}(q, \bar{q}).
\end{eqnarray}
\end{lemma}
\begin{proof}
By formulas \eqref{expa} we know that we can expand the left Cauchy kernel as
$$ S^{-1}_L(s,q)= \sum_{n=0}^{\infty} q^n s^{-1-n}.$$
Thus by Proposition \ref{res5} (which allows to exchange the operator $ \mathcal{\overline{D}}$ with the sum) and by Theorem \ref{res2} we get
\begin{eqnarray*}
\sum_{k=0}^1 q_0^k F_L(s,q)(-1)^{k+1} s^{1-k} &=& \mathcal{\overline{D}} S^{-1}_L(s,q)\\
&=& \sum_{n=0}^\infty (\mathcal{\overline{D}} q^n) s^{-1-n}\\
&=& 2 \left( \sum_{n=1}^\infty nq^{n-1}+ \sum_{j=1}^n q^{n-j} \bar{q}^{j-1}\right)s^{-1-n}.
\end{eqnarray*}
The second equality of the statement follows by applying Proposition \ref{res4} in the last equality of the previous computations. 
\\ By similar arguments it is possible to prove the equalities \eqref{P1}. 

\end{proof}
Basically, we have given two possible answers to the initial problem. Indeed, we get two possible expansions of $ \mathcal{\overline{D}}S^{-1}_L(s,q)$ and $ S^{-1}_R(s,q)\mathcal{\overline{D}}$, respectively. These will be fundamental in the next section.

\section{The polyanalytic functional calculus of order $2$ on the $S$-spectrum}
In this section we will analyse the central third row of the diagram \eqref{diagg}. From the shape of the slice hyperholomorphic Cauchy kernel, that we use to prove the integral representation (see Theorem \ref{inrap}), we have to restricted to the case of commuting operators.  

\begin{definition} \label{dbars} Let $T=T_0+\sum_{i=1}^{3} e_iT_i\in \mathcal{BC}(X)$, $s\in\hh$, we formally define the right $\overline \bigD $-kernel operator as 
$$ 2\sum_{n=1}^\infty \left(nT^{n-1}+\sum_{k=1}^n T^{n-k}\bar T^{k-1} \right) s^{-1-n} $$
and the left $\overline\bigD$-kernel operator as 
$$ 2\sum_{n=1}^\infty s^{-1-n}\left(nT^{n-1}+\sum_{k=1}^n T^{n-k}\bar T^{k-1}\right).$$
\end{definition}

Now, we recall the expansion in series of $\mathcal{F}$-resolvent operators in terms of $T$ and $ \bar{T}$, see \cite[Theorem 3.9]{CDS} with $n=3$. Let $T=T_0+\sum_{i=1}^3 e_iT_i\in \mathcal{BC}(X)$. For $s\in\hh$ with $\|T\|<|s|$ we have
\begin{equation}
\label{p1}
F_L(s,T)=-4\sum_{n=2}^\infty \sum_{\ell=1}^{n-1}(n-k)T^{n-k-1}\bar T^{k-1}s^{-1-n} \qquad F_R(s,T)=-4\sum_{n=2}^\infty \sum_{\ell=1}^{n-1}(n-k) s^{-1-n}T^{n-k-1}\bar T^{k-1}.
\end{equation}

This is fundamental for the following result.
\begin{proposition}\label{p2}
Let $T=T_0+\sum_{i=1}^3 e_iT_i\in \mathcal{BC}(X)$, $s\in\hh$ and $\|T\|<|s|$, the series in Definition \ref{dbars} converges. Moreover, we have   
\begin{equation}\label{sl2}
\sum_{j=0}^1 T_0^j(-1)^{j+1}F_{L}(s, T)s^{1-j}=2\sum_{n=1}^\infty\left(nT^{n-1}+\sum_{k=1}^n T^{n-k}\bar T^{k-1}\right)s^{-1-n}
\end{equation}
and
\begin{equation}\label{sr2}
\sum_{j=0}^1 s^{1-j}(-1)^{j+1}F_{R}(s, T) T_0^j=2\sum_{n=1}^\infty s^{-1-n} \left(nT^{n-1}+\sum_{k=1}^nT^{n-k}\bar T^{k-1}\right),
\end{equation}
where the left and right $\mathcal{F}$- resolvent operators are defined in \eqref{fres}.
\end{proposition}
\begin{proof}
First of all, we show the convergence of the series. It is sufficient to prove that the series of the operator norm:  
$$ 4\sum_{n=1}^\infty n\|T\|^{n-1}s^{-1-n}. $$
is convergent. This follows from computations similar to those in the proof of Proposition \ref{res5}.
\\ Now we prove equality \eqref{sl2}. By formulas \eqref{p1} we know how to expand in series $ F_L(s,T)$, thus we have
	$$
	\sum_{j=0}^1 T_0^j(-1)^{j+1}F_{L}(s, T)s^{1-j}=4\sum_{n=2}^\infty \sum_{k=1}^{n-1}(n-k)T^{n-k-1}\bar T^{k-1} s^{-n}-4T_0\sum_{n=2}^\infty\sum_{k=1}^{n-1}(n-k)T^{n-k-1}\bar T^{k-1}s^{-1-n}.
	$$
Now, to show  equality \eqref{sl2} is enough to prove the following equality
	\[
	\begin{split}
		& 4T_0 \sum_{n=2}^\infty\sum_{k=1}^{n-1}(n-k)T^{n-k-1}\bar T^{k-1}s^{-1-n} \\
		& = 4\sum_{n=2}^\infty \sum_{k=1}^{n-1}(n-k)T^{n-k-1}\bar T^{k-1} s^{-n} -2\sum_{n=1}^\infty nT^{n-1}s^{-1-n}-2\sum_{n=1}^\infty\sum_{k=1}^n T^{n-k}\bar T^{k-1}s^{-1-n}.
	\end{split}
	\]
	At this point, we are going to manipulate the series in the left hand side of the previous equality in order to obtain the terms in the right hand side. By using the relation: $2T_0=T+\bar T$, we obtain
	\[
	\begin{split}
		&4T_0 \sum_{n=2}^\infty\sum_{k=1}^{n-1}(n-k)T^{n-k-1}\bar T^{k-1}s^{-1-n}\\
		&=2\sum_{n=2}^\infty\sum_{k=1}^n(n-k) T^{n-k}\bar T^{k-1}s^{-1-n}+2\sum_{n=2}^\infty\sum_{k=1}^{n-1}(n-k)T^{n-k-1}\bar T^k s^{-1-n}\\
		&= 2\sum_{\ell=3}^\infty\sum_{k=1}^{\ell-1}(\ell-k-1)T^{\ell-1-k}\bar T^{k-1}s^{-\ell}+2\sum_{\ell=3}^\infty\sum_{\alpha=2}^{\ell-1}(\ell-\alpha)T^{\ell-\alpha-1}\bar T^{\alpha-1}s^{-\ell}\\
		&= 2\sum_{\ell=3}^\infty\sum_{k=1}^{\ell-1}(\ell-k)T^{\ell-1-k}\bar T^{k-1}s^{-\ell}-2\sum_{\ell=3}^\infty\sum_{k=1}^{\ell-1}T^{\ell-1-k}\bar T^{k-1}s^{-\ell}+2\sum_{\ell=3}^\infty\sum_{\alpha=1}^{\ell-1}(\ell-\alpha)T^{\ell-\alpha-1}\bar T^{\alpha-1}s^{-\ell}+\\
		&-2\sum_{\ell=3}^{\infty}(\ell-1) T^{\ell-2}s^{-\ell},\\
	\end{split}
	\]
	where in the second equality we change indexes in the first sum with $\ell=n+1$, as well as, in the second sum with $\ell=n+1$ and $k=\alpha-1$. Now, starting the first and the third series from $\ell=2$ we get
	\[
	\begin{split}
		&4T_0 \sum_{n=2}^\infty\sum_{k=1}^{n-1}(n-k)T^{n-k-1}\bar T^{k-1}s^{-1-n}\\
		&= 2\sum_{\ell=2}^\infty\sum_{k=1}^{\ell-1}(\ell-k)T^{\ell-1-k}\bar T^{k-1}s^{-\ell}-2s^{-2}-2\sum_{\ell=3}^\infty\sum_{k=1}^{\ell-1}T^{\ell-1-k}\bar T^{k-1}s^{-\ell}
		\\&+2\sum_{\ell=2}^\infty\sum_{\alpha=1}^{\ell-1}(\ell-\alpha)T^{\ell-\alpha-1}\bar T^{\alpha-1}s^{-\ell} -2s^{-2}-2\sum_{\ell=3}^{\infty}(\ell-1) T^{\ell-2}s^{-\ell}\\
		&=4\sum_{n=2}^\infty\sum_{k=1}^{n-1}(n-k)T^{nl-k-1}\bar T^{k-1}s^{-n} -2\sum_{\ell=2}^{\infty}(\ell-1) T^{\ell-2}s^{-\ell}-2\sum_{\ell=2}^\infty\sum_{k=1}^{\ell-1}T^{\ell-1-k}\bar T^{k-1}s^{-\ell}\\
		&=4\sum_{n=2}^\infty\sum_{k=1}^{n-1}(n-k)T^{n-k-1}\bar T^{k-1}s^{n} -2\sum_{\ell=1}^{\infty}n T^{n-1}s^{-n-1}-2\sum_{n=1}^\infty\sum_{k=1}^{n}T^{n-k}\bar T^{k-1}s^{-n-1},
	\end{split}
	\]
	where the last equality is obtained by the change of indexes in the second and in the third series with $n=\ell-1$. By similar arguments it is possible to prove \eqref{sr2}.
\end{proof}
\begin{corollary}
Let $T=T_0+\sum_{i=1}^3 e_iT_i\in \mathcal{BC}(X)$, $s\in\hh$ and $\|T\|<|s|$, then 
$$ \sum_{j=0}^1 T_0^j(-1)^{j+1}F_{L}(s, T)s^{1-j}=2n\sum_{n=1}^{\infty}\left(\sum_{k=0}^1 T_0^k (-1)^k (n+1-2k)Q_{n-1-k}(T,\bar T)\right) s^{-1-n} $$ 
and
$$ \sum_{j=0}^1 s^{1-j} (-1)^{j+1}F_{R}(s, T) T_0^j=\sum_{n=1}^{\infty} s^{-1-n} \left(\sum_{k=0}^1 T_0^k (-1)^k (n+1-2k)Q_{n-1-k}(T,\bar T)\right).$$
\end{corollary}
\begin{proof}
This result follows by Proposition \ref{p2} and from the fact that we can write the right $ \mathcal{\overline{D}}$-kernel operator in terms of 
$Q_{\ell}(q,\bar{q})= \frac{2}{(\ell+1) (\ell+2)}\sum_{j=0}^\ell (\ell-j+1) T^{\ell-j} \bar{T}^j$, see Proposition \ref{polydeco1}.
\end{proof}
Now, we can give the following
\begin{definition}[$\mathcal{P}_2$-resolvent operators]
Let $T=T_0+ \sum_{i=1}^3 e_i T_i \in \mathcal{BC}(X)$. For $s \in \rho_S(T)$, we define the left $ \mathcal{P}_2$-resolvent operator as
$$ \mathcal{P}_2^L(s,T)=\sum_{j=0}^1 T_0^j(-1)^{j+1}F_{L}(s, T)s^{1-j},$$
and the right $ \mathcal{P}_2$-resolvent operator as
$$ \mathcal{P}_2^R(s,T)=\sum_{j=0}^1 s^{1-j} (-1)^{j+1}F_{R}(s, T) T_0^j.$$
\end{definition}
\begin{lemma}
Let $T=T_0+ \sum_{i=1}^3 e_iT_i \in \mathcal{BC}(X)$. Then
\begin{itemize}
\item the left $ \mathcal{P}_2$-resolvent operator is a $ \mathcal{B}(X)$-valued right slice hyperholomorphic function of the variable $s$ in $\rho_S(T)$;
\item the right $ \mathcal{P}_2$-resolvent operator is a $ \mathcal{B}(X)$-valued left slice hyperholomorphic function of the variable $s$ in $\rho_S(T)$.
\end{itemize}
\end{lemma}
\begin{proof}
It follows by similar arguments of Proposition \ref{sh}.
\end{proof}
\begin{definition}[Polyanalytic functional calculus of order $2$ on the $S$-spectrum] Let $T=T_0+\sum_{i=1}^3 e_i T_i\in \mathcal{BC}(X)$ and set $ds_J=ds(-J)$ for $J\in\mathbb S$. For every function $\breve f^0=\overline \bigD f$ with $f\in \mathcal{SH}_L(\sigma_S(T))$, we set 
\begin{equation}
\label{fun1}
\breve f^0(T)=\frac 1{2\pi}\int_{\partial (U\cap\cc_J)} \mathcal{P}_2^L(s,T)\, ds_J\, f(s),
\end{equation}
where $\overline U\subset \operatorname{dom}(f)$ and $J\in\mathbb S$ is an arbitrary imaginary unit.
\\For every $\breve f^0=f\overline\bigD$ with $f\in \mathcal{SH}_R(\sigma_S(T))$, we set 
\begin{equation}
\label{fun2}
\breve f^0(T)=\frac 1{2\pi}\int_{\partial (U\cap\cc_J)} f(s)\,ds_J\,\mathcal{P}_2^R(s,T),
\end{equation}
where $U$ and $J$ are as above.
\end{definition}
By following a similar methodology developed in \cite[Theorem 4.6]{CG} we have the following result
\begin{theorem}
The polyanlytic functional calculus of order $2$ on the $S$-spectrum is well defined, i.e., the integrals \eqref{fun1} and \eqref{fun2} depend neither on the imaginary unit $J\in\mathbb S$ nor on the slice Cauchy domain $U$.
\end{theorem}

\section{Concluding remarks}
This is a seminal work about a polyanalytic functional calculus of order 2. In this paper we have introduced its definition and we have showed that it is well defined. In a forthcoming paper, we will show more properties for this functional calculus. For example, in this setting a resolvent equation holds. This is useful to prove a product rule and it is also crucial to generate the Riesz projectors. Moreover, we aim to study other polyanalytic functional calculi with orders bigger than two. Using the techniques and strategies developed in this paper, we cannot consider a polyanalytic functional calculus of order more than two.

\bibliographystyle{amsplain}

\end{document}